\documentclass{amsart}
\usepackage{fullpage}
\usepackage[latin1]{inputenc}
\usepackage{amssymb}
\usepackage{amsmath}
\usepackage{enumerate}
\usepackage{epsfig,color,caption}
\usepackage{graphicx,color}
\usepackage[curve]{xypic} 
\usepackage{hyperref}
\usepackage{graphicx}

\usepackage[dvipsnames]{xcolor}
\usepackage{colonequals, tikz}
\usepackage{tikz-cd}

\usepackage{color}

\usepackage{longtable,graphics,multirow,ulem}
\usepackage{tikz}
\usepackage{pgf,tikz}
\usetikzlibrary{arrows}
\usepackage[all]{xy}

\newcounter{nootje}
\newtheorem{theorem}{Theorem}[section]

\newtheorem{proposition}[theorem]{Proposition}
\newtheorem{corollary}[theorem]{Corollary}

\newtheorem{definition}[theorem]{Definition}
\newtheorem{example}[theorem]{Example}
\newtheorem{remark}[theorem]{Remark}

\newtheorem{rem}[theorem]{Remark}

\numberwithin{equation}{section}
\newcommand{\rk}{\mbox{rank}}

\newcommand{\ra}{\rightarrow}

\newcommand{\PGL}{\operatorname{PGL}}

\newcommand{\E}{\mathcal{E}}
\newcommand{\C }{ \mathbb{C}}

\newcommand{\Z}{\mathbb{Z}}

\usepackage{fancyhdr}
\makeatletter
\def\blfootnote{\xdef\@thefnmark{}\@footnotetext}

\author{Alice Garbagnati}
\address{Dipartimento di Matematica, Univ. Statale di Milano, Milan, Italy}
\email{alice.garbagnati@unimi.it}
\urladdr{ https://sites.google.com/site/alicegarbagnati/}
\author{Cec\'ilia Salgado}
\address{Bernoulli Institute, University of Groningen, the Netherlands}
\email{c.salgado@rug.nl}
\urladdr{ https://www.math.rug.nl/algebra/Main/salgado}

\title[Elliptic fibrations and involutions on K3s]{Elliptic fibrations and involutions on K3 surfaces}
\begin{document}


	\maketitle
	
	\begin{abstract}
	We survey our contributions on the classification of elliptic fibrations on K3 surfaces with a non-symplectic involution. We place them in the more general framework of K3 surfaces with an involution without any hypothesis on its fixed locus or on the action on the symplectic 2-form. We revisit the complete classification of elliptic fibrations on K3 surfaces with a 2-elementary N\'eron--Severi lattice, and give a complete classification of extremal elliptic fibrations on K3 surfaces that are quadratic covers of  rational elliptic surfaces.
	\end{abstract}
	\section{Introduction}

K3 surfaces occupy a prominent role in the classification of algebraic surfaces, lying in the middle of Enriques' classification. Their rich geometry has earned the attention of mathematicians in different fields, ranging from number theory, algebraic differential, and symplectic geometry to differential equations and dynamical systems. One of the implications of this rich geometry is that K3 surfaces are the only class in which surfaces might admit more than one elliptic fibration that is not of product type and has a section. Sterk showed in \cite{Sterk} that on a K3 surface there are finitely many elliptic fibrations modulo the automorphism group. It is therefore natural to search for a classification of such fibrations when they exist.

This survey aims to review our recent contributions to this problem. 

\subsection{Historical background}

Oguiso was one of the first to tackle this classification problem. In \cite{O}, he provides a classification, up to automorphisms, of all elliptic fibrations on Kummer surfaces of the product of two non-isogenous elliptic curves. In particular, he shows that, under generality conditions, such surfaces admit 11 elliptic fibrations which he describes by providing their Mordell-Weil groups and types of singular fibers (a.k.a. the frame lattice). Moreover, he shows that these are 59 distinct fibrations up to automorphisms. Oguiso's method is geometric and relies on the presence of an automorphism on $X$ which he uses as a starting point for his classification. 

A few years later, Nishiyama takes an algebraic approach and provides a classification of all possible frame lattices on K3 surfaces that are cyclic quotients of products of elliptic curves. These are Kummer or singular K3 surfaces. His approach is lattice theoretic and can be applied in a much larger generality to other K3 surfaces. 

Oguiso's techniques were explored by Kloosterman in \cite{Kl} in the particular case of K3 surfaces that are a double cover of the projective plane ramified over six lines in general position. These have Picard number 16, and moreover, the cover involution acts trivially on the N\'eron-Severi group. The presence of this involution allows him to apply Oguiso's method. 

Later, Kumar also tackled the problem by providing a classification, up to automorphisms, of all 25 elliptic fibrations on a generic jacobian Kummer surface, in \cite{Kumar}. Kumar's method is algebraic. More precisely, he exhibits candidates for divisor classes of an elliptic fibration, up to the automorphism group of the K3, in its N\'eron--Severi group. This is feasible thanks to the knowledge of the full automorphism group of the surface. 

Comparin and Garbagnati generalize Kloosterman's results, in \cite{CG}, expanding the classification for K3 surfaces that are double cover of the plane branched over the union of rational curves. There, more fiber types and higher ranks arise but the cover involution still acts trivially on the N\'eron-Severi group. As a consequence of their classification, they obtain a classification of van Geemen-Sarti involutions on such surfaces. 

Nishiyama's method was used by several authors to obtain a list of all possible frame lattices for elliptic fibrations on K3 surfaces. Examples include the K3 associated to the modular group $\Gamma_1(8)$ by \cite{BL}, and the singular K3 with transcendental lattice $\langle 2 \rangle \oplus \langle 6 \rangle$ which also admits an interpretation as a modular surface, as described in \cite{WINE1}.

Shimada adopted a different point of view on the problem providing a classification of all frame lattices of elliptic fibrations on K3 surfaces with finite Mordell--Weil groups in \cite{Shimada}. 

As Oguiso's results discussed above show, the classification of possible frames of elliptic fibrations is usually coarser than the one up to automorphisms of the surfaces. The different types of classifications of elliptic fibrations on K3 surfaces are described in \cite{BKW}. There the authors also discuss when certain classifications coincide. Recently, in \cite{FestiVeniani}, Festi and Veniani provide a method to count elliptic fibrations on K3 surfaces up to automorphisms given the knowledge of the frame lattices.

\subsection{Our contributions}

The study of K3 surfaces that can be realized as double covers of rational elliptic surfaces and the interplay between elliptic fibrations on the K3 and linear systems on the rational surface is the starting point of our series of papers, \cite{GSconcis} and \cite{GSclass}, and the \textit{Women in Numbers Europe 2 and 3} research projects \cite{WINEequ} and \cite{WINEfield}.

Our first contribution, inspired by Oguiso's work (see \cite{O} and Remark \ref{rem:Oguiso}), distinguishes the elliptic fibrations in three classes, according to the action of the cover involution on them. These yield three types of linear systems on the rational quotient. 
In \cite{GSconcis}, we describe such linear systems, which are not necessarily complete, and provide a classification of all elliptic fibrations on some explicit K3 surfaces. In particular, we show that, under certain hypotheses, the elliptic fibrations on a K3 surface induce fibrations in curves of genus 0 or 1 on the rational elliptic surface (cf. \cite[Thm. 4.2]{GSconcis}). We also describe the relationship between the singular fibers of the aforementioned fibrations on the rational surfaces and the elliptic fibrations on the K3 surfaces (cf. \cite[Thms. 5.3 and  5.5]{GSconcis}). 

The paper \cite{GSclass} puts our classification program to work by giving a complete treatment of general K3 surfaces with non-symplectic involutions that fix a curve of positive genus, since the case of a fixed locus formed only by rational curves was already analyzed in \cite{O}, \cite{Kl}, \cite{CG}. We first focus on the next case, namely of curves of genus 1, which is naturally related to rational elliptic surfaces. Indeed, the presence of a genus 1 curve in the fixed locus of a non-symplectic involution assures that the elliptic fibration determined by it on the K3 is carried on to the rational quotient. One of the features of our method is that it oftentimes naturally yields a Weierstrass equation for the elliptic fibrations. This is usually not evident and, as far as we know, most papers around the classification of elliptic fibrations do not provide a clear path. If the fixed locus of the chosen non-symplectic involution contains a curve of genus at least 2, then the relation between fibrations on the K3 and on the rational quotient is less evident. For this reason, if the non--symplectic involution fixes a curve of genus at least 2, we use Nishiyama's method to give a list of all possible frame lattices of elliptic fibrations on the K3 surface.

In \cite{WINEequ} we develop our classification program introduced in \cite{GSconcis} further by classifying elliptic fibrations on certain K3 surfaces that are not general among the ones with a non--symplectic involution (so they do not have a 2-elementary N\'eron--Severi lattice). More precisely, we give a list of all frame lattices of elliptic fibrations on the covers of the modular elliptic surface of level 5. We also describe an algorithm to determine the Weierstrass equations for many of the elliptic fibrations (see \cite[Secs. 5.2 and 6.1.2]{WINEequ}).

Our work takes an arithmetic turn in \cite{WINEfield} when we use our classification to determine fields of definition of elliptic fibrations and their Mordell-Weil groups for certain K3 surfaces. More precisely, we consider K3 surfaces that are quadratic base changes of extremal rational elliptic surfaces either with one reducible fiber or with configurations $(III^*,I_2)$ or $(III^*,III)$, branched over smooth fibers. We show that the degree of the field of definition of the fibration and the Mordell--Weil group is related to the type of elliptic fibration according to the classification introduced in \cite{GSconcis}. 

\subsection{Applications and related work} Our work has been applied in different contexts. In \cite{Demeio}, the knowledge of two distinct elliptic fibrations on a K3 provided in \cite{GSconcis} is used to show that certain classes of K3 surfaces have the Hilbert Property. In \cite{KannoWatari}, the list of elliptic fibrations given in \cite{GSclass} is used to construct a fourfold birational to a Borcea-Voisin orbifold endowed with a flat elliptic fibration morphism. The relation of our work with F-theory and string theory becomes apparent in \cite{ClingherMalmendier}. In \cite{CM}, the authors study certain K3 surfaces with a 2-elementary N\'eron-Severi lattice. For these, they provide equations for projective models using the elliptic fibrations described in \cite{GSclass} and show that the classification of frame lattices of elliptic fibrations given in \cite{GSclass} is actually a classification of all elliptic fibrations up to automorphisms applying the results of \cite{FestiVeniani}. In \cite{CM21} the authors reconsider the classification of the elliptic fibrations given in \cite{CG} and \cite{GSclass}, providing also the  classification up to automorphisms group for general members of the considered families of K3 surfaces.

\subsection{Outline of the paper}

$ $

In Section \ref{Prelim}, we cover the necessary background on fibrations on rational and K3 surfaces. 

Section \ref{K3Invol} focuses on K3 surfaces with involutions, without any extra assumption on the involution, adopting a more general approach to the problem of classifying elliptic fibrations.  We outline a classification of elliptic fibrations depending on the action of the given involution in this general setting, that restricts to the classification introduced in \cite{GSconcis} for non-symplectic involutions with non-empty fixed locus. Moreover, we determine the corresponding linear systems on the quotient surface, which in this general setting can be K3, Enriques, or rational, depending on the type of involution and on its fixed locus. We also discuss the interplay between singular fibers on the elliptic fibration on the K3 and its type with respect to the given involution. We give a more detailed account of elliptic fibrations that are preserved by the involution, dealing with those for which the involution acts on the basis in Subsec. \ref{subsec: type 2}, and those for which the involution acts only on the fibers on Subsec. \ref{subsec: type 1}. In Subsec. \ref{subsec: 2-elementary}, we restrict to K3 surfaces with a 2-elementary N\'eron--Severi lattice, recalling the results of \cite{GSclass}. 

In Section \ref{K3 fromRES}, we restrict the results from Section \ref{K3Invol} to K3 surfaces that are double covers of rational elliptic surfaces, as in our previous contributions \cite{GSconcis}, \cite{GSclass}, \cite{WINEequ} and \cite{WINEfield}. In this setting, if the involution acts only on the fibers, then quotient fibration is conic bundle (see Subsec. \ref{subsec: type 1 on RES}), and if the involution acts on the base of the fibration then the quotient fibration has genus 1 (see Subsec. \ref{subsec: type 2 on RES}). In Subsect \ref{subsec: different}, we show that a given elliptic fibration can be of different type with respect to different involutions. In fact, we provide examples of elliptic fibrations that realize simultaneously any pair of types. 

Section \ref{Sec: extremal} deals with extremal elliptic surfaces. More precisely, we list the extremal elliptic K3 surfaces that are quadratic base changes of extremal rational elliptic surfaces, whose complete classification is given in \cite{MP}. 

In Section 6, we revisit work from \cite[Sec. 5.2]{WINEequ}. More precisely, we present an algorithm to determine Weierstrass equations for certain elliptic fibrations on K3 surfaces that are double covers of rational elliptic surfaces. This is done by providing an alternative description of such K3 surfaces as double covers of minimal rational surfaces with special branch loci. 

\subsection*{Acknowledgements}
Our collaboration started in the first edition of the \textit{Women in Numbers Europe}. We are happy to thank the organizers of that and all the other \textit{Women in Numbers Europe} editions. It has been a great experience to participate in these collaborative workshops.

\section{Preliminaries}\label{Prelim}

Throughout this paper, we work over an algebraically closed field of characteristic zero, unless otherwise specified. 

\begin{definition}
Given a smooth projective algebraic surface $S$, we call a triple $(S, \pi, \sigma)$ an elliptic surface if $\pi: S \rightarrow \mathbb{P}^1$ is a relatively minimal elliptic fibration with at least one singular fiber, and a section $\sigma: \mathbb{P}^1 \rightarrow S$.  

\end{definition}
The hypothesis that $\pi$ is relatively minimal, i.e., that there are no exceptional curves among the fiber components implies that $S$ is minimal, unless $S$ is rational. In that case, the fibration $\pi$ is the unique elliptic fibration on $S$ and the negative curves of $S$ are either $(-1)$ of $(-2)$-curves. The former are sections, the latter components of reducible fibers. In what follows, rational elliptic surfaces are called RES for short.  In particular, we might refer to $S$ as a rational elliptic surface, keeping the morphisms $\pi$ and $\sigma$ implicit.
If $S$ is a K3 surface, then it might admit several relatively minimal elliptic fibrations. In that case, we make explicit which fibration we are referring to.

The reader can find more details, including examples, basic constructions and a table of all possible singular fibers in \cite{Miranda}.
\begin{remark}\label{rem: blowup description}
Given two plane cubics $F, G$ without a common component. Assume that $F$ is smooth. The blow-up of $\mathbb{P}^2$ along the possibly non-reduced scheme $F\cap G$ is a rational elliptic surface. Vice-versa, over an algebraically closed field, any RES is isomorphic to the blow-up of the projective plane in the 9, not necessarily distinct, base points of a pencil of cubics. See \cite[Lem. IV.1.2]{Miranda}.
\end{remark}
 
\textbf{Notation:} Throughout this article, the letter $R$ is reserved for rational elliptic surfaces and $X$, for K3 surfaces.

Despite not allowing for different relatively minimal elliptic fibrations, rational surfaces are, clearly, the only class of surfaces on which elliptic fibrations and fibrations on rational curves might coexist. In what follows, we introduce the definition of special rational fibrations, namely, conic bundles that fits our purposes.

\begin{definition}
Let $R$ be a rational surface. A conic bundle on $R$ is a dominant morphism $f: R \rightarrow \mathbb{P}^1$ such that its generic fiber is a smooth curve of genus 0.
\end{definition}

If $R$ is a RES and $D$ is a fiber of a conic bundle then $D\cdot (-K_R)=2$ by the adjunction formula. Hence, alternatively, by Riemann--Roch, a conic bundle is determined by an effective divisor class $D$ such that $D^2=0$ and $D\cdot (-K_R)=2$. We call the class of such $D$ in $\mathrm{NS}(R)$ a conic class, and use $D$ for the divisor or its class interchangeably.

\begin{example}\label{example: conics on  RES} \textbf{Conic bundles on RES.}
The theories of elliptic surfaces and conic bundles meet in the guise of rational elliptic surfaces. Let $(R, \pi, \sigma)$ be a rational elliptic surface. If $\pi$ has positive Mordell--Weil rank then any two sections $C_1, C_2$ that intersect transversally at a point yield a conic class $D=C_1+C_2$. Indeed, \[ D^2=C_1^2+2C_1\cdot C_2+C_2^2=-1+2-1=0 \text{ and } -K_R\cdot D= -K_R\cdot (C_1+ C_2)=1+1=2.\]
For elliptic surfaces with trivial Mordell--Weil rank, conic classes are given by sums of fiber components with sections.
Alternatively, taking as starting point the description in Remark (\ref{rem: blowup description}), any pencil of lines through one of the blown-up points yields a conic bundle in $R$.

In \cite[Figure 1]{GSconcis} and \cite{DiasCosta} one finds the description of conic classes which are sums of negative curves. 
\end{example}

\section{K3 surfaces with involutions}\label{K3Invol}

In what follows, we consider K3 surfaces with an involution. We distinguish the elliptic fibrations on the K3 according to the action of the involution on their fibers, and study the corresponding linear systems on the quotient surface.
\\

Let $X$ be a K3 surface and $\iota \in \mathrm{Aut}(X)$ an involution. Let $q: X \rightarrow X/\iota$ be the quotient map. Then one of the following holds:
\begin{itemize}
\item[i)] $X/ \iota$ is birational to a K3 surface, if $\iota$ is symplectic;
\item[ii)]  $X/ \iota$ is an Enriques surface, if $\iota$ is non-symplectic and $\mathrm{Fix}(\iota)= \emptyset$.
\item[iii)]$X/ \iota$ is a rational surface, if $\iota$ is non-symplectic and $\mathrm{Fix}(\iota)\neq \emptyset$.
\end{itemize}

The involution $\iota$ is symplectic if, and only if, $X/\iota$ is singular, cf. case i) above. In this case, we denote by $Y$ its minimal model. For simplicity, we use the letter $Y$ to denote $X/\iota$ also when this quotient is smooth.

Let $F\subset X$ be a smooth genus 1 curve. Then, by the Riemann-Roch theorem, $F$ spans a base-point-free 1-dimensional linear system $|F|$ on $X$. In other words, we have a genus 1 fibration $\varphi_{|F|}: X \rightarrow \mathbb{P}^1$. 

Let $F$ be as above. Then one of the following holds:
 \begin{itemize}
\item[a)] $\iota(F)=F$;
\item[b)] $\iota(F)=G$, for some genus 1 curve $G\neq F$ with $F\cdot G=m \geq 0$.
\end{itemize}

Let $D=q(F)$ then $D$ is a curve of genus 0 or 1. The following proposition gives a characterization of the case $g(D)$=0.

\begin{proposition}\label{prop:g(D)}
	Let $X$ be a K3 surface, $\iota \in \mathrm{Aut}(X)$ an involution, $F, D$ and $q$ as above. Then the following are equivalent:\begin{itemize}
		\item[a)] 
 $g(D)=0$; 	 \item[b)] $F\notin Fix(\iota)$, $\iota(F)=F$ and $\iota_{|F}$ has 4 fixed points;
\item[c)] either $\iota$ is non-symplectic and $\iota(E)=E$ for each $E\in |F|$ or $\iota$ is symplectic and there exist $E, E'\in |F|$ such that $\iota(E)=E'$ and $E\neq E'$.\end{itemize}
\end{proposition}	
	\begin{proof} To prove the equivalence between a) and b), we observe that if $\iota(F)=G$ where $F\neq G$, then $q(F)\simeq q(G)\simeq F$, hence $g(D)=1$. So we may assume that $\iota$ acts on $F$, i.e. $\iota(F)=F$, and $D=q(F)$ is $F/\iota_{|F}$. In particular $g(D)=0$ if and only if $\iota_{|F}$ is not the identity and is not fixed point free. Since $\iota$ is an involution, this is equivalent to $\iota_{|F}$ fixes 4 points. 
	
	We show that a) implies c).	
	If $g(D)=0$, then  $\iota(F)=F$, so the linear system $|F|$ is preserved by $\iota$. Therefore either each member of the linear system is preserved or it is mapped to another member of the same linear system.
	In the first case there exists $E\in |F|$ such that $\iota_{|E}\neq id_{|E}$ (otherwise $\iota$ would be the identity on $X$) and $\iota$ cannot be a translation on the fibers, since $g(D)=0$. So $\iota$ is the elliptic involution on each fiber (possibly composed with a translation). This implies that $\iota_{|E}$ does not preserve the period of $E$, and hence it does not preserve the period of $X$, i.e. it is non-symplectic.
	Otherwise, there exist $E,\ E'\in |F|$ such that $\iota(E)=E'$, so $\iota$ is not the identity on the basis of the fibration and it preserves exactly two fibers: $F$ is one of them. We already observed that $\iota_{|F}$ is the elliptic involution (possibly composed with a translation), so $\iota$ is the composition of two non-symplectic involutions, one acting on the basis, the other on the fibers. This implies that $\iota$ is symplectic. 
	
	Viceversa $c)$ implies a). Indeed if $\iota$ is non-symplectic and it is the identity on $|F|$, it is the elliptic involution (possibly composed with a translation) on each fiber, so $D\simeq F/\iota_{|F}$ is a rational curve. If $\iota$ is symplectic and acts as a non-trivial involution on the basis, then $\iota$ is the composition of an involution on the basis and a non-period preserving involution on the fiber. This implies that $\iota_{|F}$ is the elliptic involution and $D\simeq F/\iota_F$ is rational. 
	\end{proof}

\begin{rem}\label{rem: linear systems B and D symplectic} We observe that in the case $\iota$ is symplectic and $g(D)=0$, the surface $X/\iota$ is singular in 8 points and $D$ passes through 4 of them. We consider the divisor $B=\beta^*(D)$ where $\beta:Y\ra X/\iota$ is the minimal resolution. If $E\in |F|$ is a curve not preserved by $\iota$, then $\beta^*(D)\sim \beta^*(q(E))$, and since $q(E)$ does not pass through the singular points of $X/\iota$, $\beta^*(D)^2=\beta^*(q(E))^2=q(E)^2=F^2=0$. In particular $|B|$ defines a fibration on $Y$.\end{rem}

\begin{remark}
If $F\subset \mathrm{Fix}(\iota)$, i.e,  $\iota|_{F}=id|_F$, then $\iota$ is non-symplectic, since it fixes a curve and $D\cong F$.
\end{remark}

\begin{proposition}\label{prop:multisection}
Let $X$ be a K3 surface, $\iota \in \mathrm{Aut}(X)$ an involution and $F\subset X$ a smooth genus 1 curve. If $\iota(F)=G$ and $m>0$, as in b) above, then there is at least one curve $M \in \mathrm{Fix}(\iota)$. In particular, $\iota$ is non-symplectic. 
\end{proposition}
\begin{proof}
Let $F_t\in |F|$ and $G_t\in |G|$ such that $\iota(F_t)=G_t$. Then $F_t\cap G_t \in \mathrm{Fix}(\iota)$ and it consists of $m$ points which we denote by $P_1^{(t)}, \cdots, P_m^{(t)}.$  Let $s\neq t$, then $F_s\cap F_t=\emptyset$. Hence the sets $\{P_1^{(t)}, \cdots, P_m^{(t)}\} $ and $ \{P_1^{(s)}, \cdots, P_m^{(s)}\}$ are disjoint. In particular, varying $t \in \mathbb{P}^1$, yields infinitely many points in $\mathrm{Fix}(\iota)$. Since $\mathrm{Fix}(\iota)$ is a union of finitely many proper subvarieties, it contains a curve, which we call $M$. \end{proof}

\begin{remark}
If $m>0$, then the curve $M$ from the Prop. \ref{prop:multisection} is a multisection of degree $m$ of $\varphi_{|F|}$ and $\varphi_{|G|}$. Indeed, $M \cdot F_t =m=M \cdot G_t, \forall F_t\in |F|, G_t \in |G|$.
\end{remark}

\begin{remark}
If $\iota$ is an Enriques involution then $\iota(F)\in |F|$. In particular, if $\iota(F)=G$ then $m=0$. Moreover, if $\iota(F)=F$ then $\iota|_{F}$ is a translation. 
\end{remark}

\begin{definition}\label{def: class}
 Let $\iota$ be an involution on a K3 surface $X$ and  $\varphi_{|F|}: X \rightarrow \mathbb{P}^1$ be as above. We say that $\varphi_{|F|}$ is of
\begin{itemize}
\item[Type 1)] with respect to  $\iota$, if $\iota(F_t) = F_t$, $\forall t \in \mathbb{P}^1$, i.e., $\iota$ acts as an involution in $F_t$;
\item[Type 2)] with respect to $\iota$, if $\iota(F_t)=F_s$, for $F_s:=\varphi_{|F|}^{-1}(s)$;
\item[Type 3)] with respect to $\iota$, if  $\iota(F_t) = G_t$, for $G_t$ a genus 1 curve which is not a fiber of $\varphi_{|F|}$.
\end{itemize} 
\end{definition}

\begin{remark}\label{rem:fibration}
If $\varphi_{|F|}$ is of type 1 or 2 with respect to $\iota$, then
\begin{itemize} 
\item[i)] $\iota$ preserves the class of a fiber of $\varphi_{|F|}$. If it is of type 1 then $\iota$ acts trivially on the base of  $\varphi_{|F|}$. If it is of type 2 then it acts as an involution of the base of $\varphi_{|F|}$, and in particular, $\iota$ preserves two fibers of  $\varphi_{|F|}$;
 \item[ii)] if $X/\iota$ is smooth (i.e. $\iota$ is non-symplectic), and $D=q(F)$ then $\varphi_{|D|}: Y \rightarrow \mathbb{P}^1$ is a fibration. Indeed, $D^2=0$ since the class of the divisor $q^*D$ is a multiple of the class of $F$.
 \item[iii)] if $X/\iota$ is singular (i.e. $\iota$ is symplectic) the linear system $B$ defines a fibration on the desingularization of $Y$, see Remark \ref{rem: linear systems B and D symplectic}.
 \end{itemize}
\end{remark}

\begin{remark}\label{rem:Oguiso}
	The classification from Definition \ref{def: class} was inspired by Oguiso's work \cite{O}. More precisely, Oguiso considered K3 surfaces with a non-symplectic involution acting trivially on the N\'eron--Severi group and distinguished the elliptic fibrations according to the presence of a section in the fixed locus of the involution. The presence of a section on the fixed locus implies that the involution acts trivially on the basis of the elliptic fibration. These are what we called fibrations on type 1 in Def.\ref{def: class}. For remaining fibrations in \cite{O}, the involution gives an involution on the basis. These are called of type 2 in Def.\ref{def: class}. Finally, we observe that there is no analogous of a fibration of type 3 in \cite{O} thanks to the trivial action of the involution on the N\'eron--Severi group.
\end{remark}

We investigate the nature of $Y$ and $\varphi_{|D|}$ in each of the cases described above. 

\begin{corollary}\label{cor: type1 on quotient} If $\varphi_{|F|}$  is of type 1 with respect to $\iota$ then $\varphi_{|D|}:Y \rightarrow \mathbb{P}^1$ is a fibration and $Y$ is either K3 or rational. If $Y$ is K3 then $|D|$ is a genus 1 fibration on $Y$. If $Y$ is rational then $|D|$ is a conic bundle on $Y$.
\end{corollary}
\begin{proof}
By Remark \ref{rem:fibration}, $\varphi_{|D|}$ is a fibration on $Y$ and the genus of its generic fiber is given in Proposition \ref{prop:g(D)}. It remains to show that $\iota$ has a non-empty fixed locus. By i) in Remark \ref{rem:fibration}, $\iota$ acts as the identity on the base of $\varphi_{|F|}$. If $\iota|_{F_t}$ acts as a translation of a genus 1 curve for all the smooth fibers $F_t$, then it preserves their period and so it is symplectic. Otherwise $\iota_{F_{t}}$ has some fixed points on  $F_t$. In other words, $\iota$ is symplectic or it is non-symplectic with $\mathrm{Fix}(\iota)\neq \emptyset$. 
\end{proof} 

\begin{corollary}\label{cor:type2} If $\varphi_{|F|}$  is of type 2 with respect to a non-symplectic involution $\iota$ then $\varphi_{|D|}:Y \rightarrow \mathbb{P}^1$ is a genus 1 fibration. 
If $\varphi_{|F|}$  is of type 2 with respect to a symplectic involution $\iota$ then $\varphi_{|B|}:Y \rightarrow \mathbb{P}^1$ is a genus 1 fibration (where $B$ is as in Remark \ref{rem: linear systems B and D symplectic}). 
\end{corollary}
\begin{proof}
This is a combination of i) in Remark \ref{rem:fibration} and Proposition \ref{prop:g(D)}. We observe that if $\iota$ is symplectic, then $Y$ is a K3 surface and each fibration on it is necessarily a genus 1 fibration. 
\end{proof}

\begin{corollary}\label{cor:type3}  If $\varphi_{|F|}$  is of type 3 with respect to $\iota$ then $\dim |D|>1$ and $\iota$ is a non-symplectic involution with non-empty fixed locus. In particular, $Y$ is a rational surface.
\end{corollary}
\begin{proof}
If $\varphi_{|F|}$ is of type $3$ then $F_t \cdot G_t =m>0$. By Prop. \ref{prop:multisection}, there is a curve $M\subset \mathrm{Fix}(\iota)$. Hence $\iota$ is non-symplectic with a non-empty fixed locus. 
For the dimension of $|D|$, observe that $q^*(D)=F+G$ and $2D^2=(q^*(F+G))^2=2m>0$. Hence $D^2=m>0$ and $|D|$ does not induce a fibration on $Y$.\end{proof}

\subsection{Fibrations of type 2}\label{subsec: type 2}
By Corollary \ref{cor:type2}, fibrations of type 2 with respect to $\iota$ induce genus 1 fibrations on $Y$. We have a closer look at such fibrations in what follows.
Let $q:X \rightarrow Y$ be the quotient map, possibly composed with a birational map. By an abuse of notation we denote $|D|$ both the linear system $|D|$ when $X/\iota$ is smooth and the linear system $|\beta^*D|=|B|$ when $\beta:Y\ra X/\iota$ is the desingularization of $Y$. We have the following commutative diagram

$$\xymatrix{X \ar@{-->}[rr]^q_{2:1} \ar[d]_{\varphi_{|F|}} &&Y \ar[d]^{\varphi_{|D|}}  \\ \mathbb{P}^1 \ar[rr]^{\psi}_{2:1}&& \mathbb{P}^1}
$$
where $\psi$ is the map induced by $q$ on the base of $\varphi_{|F|}$. In particular, $\psi$ is branched above two points $t_0,t_1 \in \mathbb{P}^1$. Let $D_{t_0}$ and $D_{t_1}$ be the fibers of $\varphi_{|D|}$ above them. Then the restriction $q|_{F_t{_i}}: F_{t_i} \rightarrow D_{t_i}$ is $1:1$, if, and only if, $D_{t_i}$ is in the branch locus of $q$, or equivalently, if, and only if, $D_{t_i}$ is not a multiple fiber of $\varphi_{|D|}$.

\begin{proposition}\label{prop: type 2 fibrations}
With the notation as above, let $\varphi_{|F|}$ be a fibration of type 2 with respect to $\iota$. Then the following holds.
\begin{itemize}
\item[i)] If $Y$ is K3 then $\varphi_{|D|}$ has no multiple fibers, and at least one among $D_{t_0}$ and $D_{t_1}$ is reducible, with a non-reduced component.
\item[ii)] If $Y$ is Enriques then both $D_{t_0}$ and $D_{t_1}$ are multiple fibers.
\item[iii)] If $Y$ is rational then at most one among $D_{t_0}$ and $D_{t_1}$ is a multiple fiber.
\end{itemize}
\end{proposition}
\begin{proof}
First observe that $\mathrm{Fix}(\iota)\subset F_{t_0}\cup F_{t_1}$. If $\mathrm{Fix}(\iota)=\emptyset$ then $\iota$ is an Enriques involution and the branch locus of $q$ is empty. Hence $D_{t_0}$ and $D_{t_1}$ are multiple fibers of $\varphi_{|D|}$. This is case ii). Assume that $\mathrm{Fix}(\iota)\neq \emptyset$. If $\iota$ is symplectic then $\iota|_{F_{t_i}}$ fixes 4 points, for $i=0,1$. Hence $q(F_{t_i})$ is a singular curve of multiplicity 2. After resolving its singularities, we obtain a non-multiple fiber $D_{t_{i}}$ with non-reduced components. If $\iota$ is non-symplectic, then at least one among $F_{t_{0}}$ and $F_{t_{1}}$, say $F_{t_0}$ has some components contained in $\mathrm{Fix}(\iota)$. Hence $D_{t_{0}}$ shares a common component with the branch locus of $q$ and, in particular, it cannot be a multiple fiber.
\end{proof}

Under the assumption that $\varphi_{|D|}$ has a section then the statement of Proposition \ref{prop: type 2 fibrations} becomes simpler.

\begin{corollary}\label{cor: type 2 with section}
With the notation as above, let $\varphi_{|F|}$ be a fibration of type 2 with respect to $\iota$. Suppose that $\varphi_{|D|}$ has a section. Then 
\begin{itemize}
\item[i)] $Y$ is not an Enriques surface;
\item[ii)] $D_{t_0}$ and $D_{t_1}$ are both branch curves for $q$;
\item[iii)] $F_{t_i} \in \{I_{2n}, IV, I_0^*, IV^*\}$, for $i=0, 1$; 
\item[iv)] If $Y$ is rational then $D_{t_0}$ and $D_{t_1}$ are reduced fibers of $\varphi_{|D|}$.
\end{itemize}
\end{corollary}
\begin{proof}
The presence of a section prevents the existence of multiple fibers, hence case ii) of Prop. \ref{prop: type 2 fibrations} cannot happen and $D_{t_i}$ are both branch fibers. This implies that $F_{t_i}$ are both obtained by a quadratic base change on possibily sigular fibers of an elliptic fibration, which proves iii).
Moreover, the presence of a section of $\varphi_{|D|}$, also assures that $\varphi_{|F|}$ admits a section.
If $Y$ is rational, and one of $D_{t_i}$ is non-reduced, then an Euler number computation, implies that the Euler characteristic $e(X)=12$ or $0$, and hence $X$ would not be a K3. 
\end{proof}

\begin{remark}
Corollary \ref{cor: type 2 with section} allows us to identify candidates for fibrations of type 2 with respect to a given involution among the list of all elliptic fibrations on a given K3 surface. Indeed, any fiber type which is not of the type listed in iii) of Corollary \ref{cor: type 2 with section} occurs in pairs, and there are at most two fiber types that occur alone.
\end{remark}

\subsection{Fibrations of type 1}\label{subsec: type 1}
Given an involution $\iota \in \mathrm{Aut}(X)$. We discuss the properties of fibrations of type 1 with respect to $\iota$. We first deal with symplectic involutions and later with non-symplectic ones.

\begin{proposition}\label{pro:2-torsion}
Let $X$ be a K3 surface and $\iota \in \mathrm{Aut}(X)$ a symplectic involution. Let $\varphi_{|F|}$ be an elliptic fibration on $X$ of type 1 with respect to $\iota$, and $\varphi_{|D|}$ the respective elliptic fibration on $Y$. Then both $\varphi_{|F|}$ and $\varphi_{|D|}$ admit a $2$-torsion section. In particular, both fibrations cannot have fibers of types $II^*$ or $IV^*$.
\end{proposition}
\begin{proof}
Since $\varphi_{|F|}$ is of type 1 with respect to $\iota$, this involution acts as the translation by a $2$-torsion point in each fiber of $\varphi_{|F|}$. Hence it acts as a translation by a $2$-torsion point on the generic fiber, which clearly implies the presence of a $2$-torsion section for $\varphi_{|F|}$. In particular, we can write a (generalized) Weierstrass equation for $\varphi_{|F|}$ as follows:
\[
y^2=x(x^2+a(t)x+b(t)).
\]
Taking the quotient by the 2-torsion section $(x,y,t)=(0,0,t)$ yields a (generalized) Weierstrass equation for $\varphi_{|D|}$ (see \cite[Proposition 3.7]{ST}):
 \[
y^2=x(x^2-2a(t)x+(a(t)^2-4b(t)),
\]
which clearly admits a $2$-torsion section of the same shape.
For the statement on the fiber types, recall that fibers of type $II^*$ and $IV^*$ are incompatible with the presence of a $2$-torsion section (\cite[Table VII.2.6.]{Miranda}).
\end{proof}
By Prop. \ref{pro:2-torsion}, $\iota$ is a translation by a 2-torsion section, which we denote by $T$. The fiber types of $\varphi_{|D|}$ can be readily determined by the Weierstrass equation given in the proof. For completion, we present them below. We adopt the standard notation for the fiber components, namely, $\Theta_i$ denote the components and $\Theta_0$, the neutral component, i.e., the one met by the zero section.
\begin{corollary}
Let $X, \iota, q, Y, \varphi_{|F|}$, and $\varphi_{|D|}$  be as above. Then the fiber types of $\varphi_{|F|}$ and $\varphi_{|D|}$ are related as follows.
\small{
$$
\begin{array}{|c|c|c|c|c|c|c|c|}
\hline
F_t & III^*& I_{2n+1} & I_{2n} , T\cdot \Theta_0=1& I_{2n} , T\cdot \Theta_n=1 &  I^*_{2n+1} & I^*_{2n} , T \cdot \Theta_{\text{near}}=1& I^*_{2n} , T\cdot \Theta_{\text{far}}=1 
\\
\hline
D_{\psi(t)} & III^* & I_{4n+2} & I_{4n} & I_{n}  &  I^*_{4n+2} & I^*_{4n} & I^*_{n}
\\
\hline
\end{array}
$$}
\end{corollary}

We move to study fibrations of type 1 with respect to a given non-symplectic involution. In parallel to the above, we first obtain results on the structure of the Mordell--Weil group and inadmissible singular fibers. Finally, we provide the reader with references for an account of the admissible types of singular fibers. 

\begin{proposition}\label{non-symplectic_Type1}
Let $X$ be a K3 surface $\iota \in \mathrm{Aut}(X)$ a non-symplectic involution and $\varphi_{|F|}$ a fibration of type 1 with respect to $\iota$.
\begin{itemize}
\item[i)] If $\iota^*$ acts trivially on $NS(X)$, then $MW(\varphi_{|F|})\subset(\Z/2\Z)^2$;
\item[ii)] If $\iota$ fixes two smooth genus 1 curves, then $\varphi_{|F|}$ has at most 6 irreducible singular fibers and it admits no fibers of type $II^*, IV^*, II, III, IV, I_n$, for $n$ odd and $n>7$, and $I_m^*$, for $m$ odd and $m>5$ or $m$ even and $m>12$.
\end{itemize}
\end{proposition}
\begin{proof}
For part 1), see \cite[Prop. 2.5]{GSclass} and \cite[Lemma 2.4]{O}. Part ii) follows from \cite[Remarks 5.6 and 5.7]{GSconcis}.
\end{proof}

\begin{remark}
The reader will find a table for all admissible singular fibers of $\varphi_{|F|}$ in \cite[Table 1]{GSconcis}. This is also under the hypothesis as in ii) Prop. \ref{non-symplectic_Type1}.
\end{remark}

\begin{remark}
The reader might notice the presence of extra assumptions on $\mathrm{Fix}(\iota)$ and on the action of $\iota^*$ on $\mathrm{NS}(X)$ in the statements for the non-symplectic case. These were not present in the symplectic case as there the fixed locus is always given by 8 points and $\iota^*$ never acts trivially on $\mathrm{NS}(X)$.
\end{remark}

\subsection{Non-symplectic involutions acting trivially on $\mathrm{NS}(X)$.}\label{subsec: 2-elementary}
The assumption that the action of $\iota^*$ on $\mathrm{NS}(X)$ is trivial simplifies the study of elliptic fibrations on $X$. Indeed, not only do we have the previous results at hand, but moreover, there are no fibrations of type 3 with respect to $\iota$. We summarize our results on such special K3 surfaces in what follows.

\begin{proposition}\label{prop: 2 elementary 0}
Let $X$ be a K3 surface. Then there is a non-symplectic involution $\iota\in \mathrm{Aut}(X)$ such that $\iota^*$ acts trivially on $\mathrm{NS}(X)$ if, and only if, $\mathrm{NS}(X)$ is $2$-elementary. In that case, $X$ does not admit elliptic fibrations of type 3 with respect to $\iota$, and all elliptic fibrations on $X$ are classified.
\end{proposition}
\begin{proof}
The first assertion is classic and proved by Nikulin in \cite{Nikulin} (see also \cite[Thm. 0.1]{AST}). The nonexistence of fibrations of type 3 follows directly from its definition. Finally, the classification of elliptic fibrations on such surfaces is given in a series of papers by different authors. More precisely, the 2-elementary lattices that can be realized as the  N\'eron--Severi lattice of a K3 are determined by 3 invariants, namely, $r$, its rank, $a$, its length, and $\delta$, which encodes information on the discriminant form. If $(r,a,\delta)=(18,4,0)$, the K3 surface is the Kummer surface of two non-isogenous elliptic curves and the classification was given by Oguiso in \cite{O}. If $(r,a,\delta)=(16,6,1)$, the K3 surface is the double cover of the plane branched along 6 lines in general position, and the classification was given by Kloosterman in \cite{Kl}. More generally, if $r+a=22$, then the fixed locus of the involution consists only of rational curves and the classification was given in \cite{CG}. The remaining cases are contained in \cite{GSclass}. We observe that the fixed locus contains a curve of genus $g=\frac{22-r-a}{2}$. In particular, if $r+a=22, \text{ or }20$, then the K3 can be realized as a double cover of a RES, branched along two fibers, which might be reducible or not, depending on $r+a$. \end{proof}

\begin{proposition}\label{prop:2-elementary}
Let $X$ be a K3 surface. Assume that $\mathrm{NS}(X)$ is 2-elementary. Then there is an involution $\sigma \in \mathrm{Aut}(X)$ such that there are no elliptic fibrations on $X$ of type 3 with respect to $\sigma$, and every elliptic fibration of type 1 has a finite Mordell--Weil group. In particular, if there is an elliptic fibration with infinite Mordell--Weil group then it is of type 2 with respect to $\sigma$.
\end{proposition}
\begin{proof}
Under the hypothesis, there exists an involution $\sigma$ such that $\sigma^*$ acts trivially on $\mathrm{NS}(X)$ (as in Prop. \ref{prop: 2 elementary 0}). If $\nu$ is an elliptic fibration of type 1 with a section $C$ of infinite order then $C$ is not preserved by $\sigma$ which contradicts the triviality of $\sigma^*$. 
\end{proof}

\section{K3s that are base change of RES}\label{K3 fromRES}

In the previous sections, our input data has been that of a K3 surface with an involution. In our series of papers \cite{GSconcis}, \cite{GSclass}, \cite{WINEequ} and \cite{WINEfield} we adopted a different viewpoint by considering a pair $(X,R)$, with $R$ a RES and $X$ a K3 surface obtained as a double cover of $R$. This is equivalent to fixing a K3 surface with a non-symplectic involution whose fixed locus is composed of curves of genus at most 1.

In what follows, we give a brief account of our contributions and place them in the light of the previous sections. 

Let $\mathcal{E}_R: R \rightarrow B_0=\mathbb{P}^1$ be a RES and $\psi: B=\mathbb{P}^1  \rightarrow  B_0$, a degree 2 morphism. We have the following commutative diagram:

\begin{equation}\label{diag base chage}\xymatrix{ X \ar[r]\ar[dr]_{\mathcal{E}_X} & R\times_{\psi}\mathbb{P}^1 \ar[rr]^q_{2:1} \ar[d] &&R \ar[d]^{\mathcal{E}_R}  \\ & \mathbb{P}^1 \ar[rr]^{\psi}_{2:1}&& \mathbb{P}^1}
\end{equation}
where $X$ is a minimal model of $R\times_{\varphi}\mathbb{P}^1$. If  $\psi$ is branched over reduced fibers of $\mathcal{E}_R$, then $X$ is a K3 surface, and in that case, it is the minimal model of $R\times_{\psi}\mathbb{P}^1$. We assume that this is the case from now on.

The elliptic fibration $\mathcal{E}_R$ naturally induces an elliptic fibration on $X$, which is denoted by $\mathcal{E}_X$ above and in what follows. 

\begin{remark}
If the branch fibers of $\psi$ are smooth then $R \times_{\psi} \mathbb{P}^1$ is smooth and $X$ coincides with it.
\end{remark}

The degree 2 map $\psi: B\rightarrow B_0$ induces an involution on $B$ which lifts to the deck involution of the cover $X\rightarrow R$. It clearly preserves the class of a fiber of $\mathcal{E}_X$ in $\mathrm{NS}(X)$. We denote this involution by $\iota$ in what follows. 

\textbf{Context:} Let $X$ be a K3 surface constructed as above and $\iota$ be the deck involution. In what follows, we analyze the elliptic fibrations on $X$ with respect to $\iota$.

We start by observing that $\mathrm{Fix}(\iota)$ is contained in the union of the two fibers over the ramification points of $\psi$. Hence it consists of a union of curves of genus at most 1. In particular, a genus 1 curve contained in it must be a fiber of $\mathcal{E}_X$.

\begin{corollary}
The elliptic fibration $\mathcal{E}_X$ is of type 2 with respect to $\iota$. Moreover, if $\mathrm{Fix}(\iota)$ contains at least one smooth fiber of $\mathcal{E}_X$, then $\mathcal{E}_X$ is the unique fibration of type 2 with respect to $\iota$.
\end{corollary}
\begin{proof}
Let $F$ be the smooth fiber of $\mathcal{E}_X$ contained in the fixed locus. If $\nu:X\ra\mathbb{P}^1$ is a fibration of type 2 with respect to $\iota$, the fixed locus of $\iota$ is contained in fibers of $\nu$ and if it contains a genus 1 curve, then it is a fiber of the fibration. This implies that $\mathcal{E}_X$ and $\nu$ coincides, since they share a fiber.  
	\end{proof}

If $\mathrm{Fix}(\iota)$ contains only rational curves, then there could be other elliptic fibrations of type 2 with respect to $\iota$. This gives other realizations of $X$ as a base change of a RES. Indeed, in that case, $R \times_{\psi} \mathbb{P}^1$ is singular and $X$ is a blow-up of $R\times_{\psi} \mathbb{P}^1$. Hence there is a commutative diagram

\begin{equation}\label{diag quotient} \xymatrix{ X \ar[r]^{f}\ar[d]_{q}& R\times_{\psi}\mathbb{P}^1 \ar[d]  \\Y = X / \iota \ar[r]& R}
\end{equation}
where $q$ is the quotient map, $f$ is a birational morphism and $Y$ is a rational surface with $K_Y^2<0$. In particular,  being a blow-up of $R$, the surface $Y$ might have more than one non-relatively minimal elliptic fibrations. In that case, there is more than one way to realize $X$ as a double cover of a RES and, by construction, each non-relatively minimal elliptic fibration on $Y$ yields a fibration of type 2 on $X$. 
In what follows, we denote by $\tilde{f}$ and $\tilde{q}$ the maps in the commutative diagram (\ref{diag quotient}) such that $\tilde{q}\circ f = \tilde{f} \circ q $.

\subsection{Fibrations of type 1}\label{subsec: type 1 on RES}

There are always fibrations of type 1 on K3 surfaces that arise as a double cover of a RES. Indeed, as discussed in Example \ref{example: conics on  RES}, RESs admit conic bundles, and their pull-back yield elliptic fibrations on the K3 surface. The following is a reciprocal of Corollary \ref{cor: type1 on quotient} restricted to our context of K3 surfaces that are base change of RES.
\begin{proposition}\label{prop: type 1 CB}
Let $Y$ be as above and $|D|$ a conic bundle on $Y$.  Then $|q^{-1}(D)|$ is a genus 1 fibration on $X$. Moreover, it is of type 1.
\end{proposition}
\begin{proof}
It is enough to notice that $q^*(D)^2=D^2=0$. By adjunction $q^*(D)$ has genus 1 and its class spans a genus 1 fibration on $X$. It is of type 1 by Corollaries \ref{cor:type2} and \ref{cor:type3}.
 \end{proof}

If $D \subset Y$ is a genus zero curve such that $D^2=0$ and $D\cdot (-K_{Y})=2$, then $D$ is a conic class in $Y$ and hence it endows $Y$ with a conic bundle structure. The linear system $|\tilde{f}(D)|$ in $R$ has base points if, and only if, $D$ intersects at least one of the exceptional curves of the blow-up map $\tilde{f}$. In this case, $|\tilde{f}(D)|$ is clearly not a conic bundle on $R$. If the latter holds, we called the divisor class $\tilde{f}(D)$ in $R$ a \emph{generalized conic bundle} in \cite{GSconcis},  to stress the fact that $D$ is a conic bundle only on $Y$. We reformulate Proposition \ref{prop: type 1 CB} and its reciprocal in that language. 

\begin{proposition} \cite[Proposition 3.8 and Theorem 4.2]{GSconcis}
Let $X$, $R$ and $\iota$ be as above. Let $\pi: X \rightarrow \mathbb{P}^1$ be an elliptic fibration. Then $\pi$ is of type 1 with respect to $\iota$ if, and only if, $|\tilde{f}\circ q (\pi^{-1}(t))|$ is a conic bundle or a generalized conic bundle on $R$. 

\end{proposition}

\subsection{Fibrations of type 2}\label{subsec: type 2 on RES}
While $R$ admits only one elliptic fibration, say $\pi$, the quotient surface $Y$, if different from $R$, might admit several genus 1 fibrations with a section that are not relatively minimal. Given such a genus 1 fibration $\nu: Y \rightarrow \mathbb{P}^1$ different from $\pi$, the blow-down of $Y$ to a surface on which $\nu$ is relatively minimal, say $R'$, produces a different realization of $X$ as a double cover of a different RES. Similarly to what happens for generalized conic bundles, the curves $\tilde{f}(\nu^{-1}(s))$ span a linear system that has base points on $R$ and in particular is not a fibration. In \cite{GSconcis} and \cite{GSclass} we called such linear systems \emph{splitting genus 1 fibrations} on $R$.

We can now reformulate Corollary \ref{cor:type2} and its reciprocal in this language.
\begin{proposition}
Let $X$, $R$ and $\iota$ be as above. Let $\nu: X \rightarrow \mathbb{P}^1$ be an elliptic fibration. Then $\nu	$ is of type 2 with respect to $\iota$ if, and only if, the linear system  $|\tilde{f}\circ q(\nu^{-1}(s))|$ is a (splitting) genus 1 fibration on $R$. 
\end{proposition}

\begin{remark}
Let $X$ be a K3 with  2-elementary $\mathrm{NS}(X)$. Assume that there is an elliptic fibration $\varepsilon: X \rightarrow \mathbb{P}^1$ with infinite Mordell--Weil group. Prop. \ref{prop:2-elementary} implies that there exists a RES $\varepsilon_R: R \rightarrow \mathbb{P}^1$ such that $\varepsilon$ is a quadratic base change of $\varepsilon_R$.
\end{remark}

\subsection{Examples of fibrations that are of different types with respect to different involutions}\label{subsec: different}

$ $

Given a K3 surface $X$ as in diagram \eqref{diag base chage} and its deck involution $\iota$, we observed that any elliptic fibration is of type 1,2 or 3 with respect to $\iota$. The type of a fibration is, of course, relative to the involution we are considering. Hence, a priori, a given elliptic fibration on $X$ can be of different types with respect to different involutions. A natural question is whether two different types can be simultaneously realized for the same elliptic fibration on a given K3 surface. We provide some examples, gathered from our previous papers \cite{WINEfield} and \cite{WINEequ} that give a positive answer to this question.

{\bf A Fibration realized as types 2 and 3.} Let us consider the K3 surface $X$ whose transcendental lattice is $U\oplus U(2)$. The frame lattices of the elliptic fibrations on $X$ are listed in \cite{GSclass}, and \cite{CM21} shows that this list also gives a classification of elliptic fibrations on $X$ up to automorphisms. In \cite{WINEfield} the surface $X$ is obtained as a quadratic base change of two different elliptic fibrations: $R_3$, whose singular fibers are $III^*+I_2+I_1$ and $R_4$, whose singular fibers are $I_4^*+2I_1$. In both cases the branch fibers are smooth and the deck involution of $X\ra R_i$ will be denoted by $\iota_i$, for $i=3, 4$. 

In \cite[Tables 5 and 6]{WINEfield} the type of each elliptic fibration on $X$  with respect to $\iota_3$ and $\iota_4$ are determined. In particular, by \cite[Tables 5 and 6]{WINEfield} one observes that $X$ admits an elliptic fibration whose singular fibers are $2III^*+2I_2+2I_1$. It is of type 2 with respect to $\iota_3$ and of type 3 with respect to $\iota_4$. 

{\bf A Fibration realized as types 1 and 3.} Consider the same surface and involutions as above. Then it admits another fibration whose singular fibers are $I_{12}^*$. It is of type 3 with respect to $\iota_3$ and of type 1 with respect to $\iota_4$.

{\bf A Fibration realized as types 1 and 2.} In \cite[Table 6.1]{WINEequ} the frame of the elliptic fibrations on a surface denoted by $S_{5,5}$ are listed, and by \cite{CM21} this gives a classification of the elliptic fibration up to automorphisms. The surface $S_{5,5}$ is obtained by a base change on a RES whose singular fibers are $2I_5+2I_1$ and the branch fibers are the $I_5$-fibers. Let us denote by $\iota_{5,5}$ the deck involution. In \cite[Table 6.1]{WINEequ}, we observe that on $S_{5,5}$ there is a fibration with singular fibers $2I_4^*+2I_2$ and that it is of type 1 with respect to $\iota_{5,5}$. The surface $S_{5,5}$ is the unique K3 surface with transcendental lattice $\left[\begin{array}{cc}2&0\\0&2\end{array}\right]$ and coincides with the surface that one finds as a quadratic base change of a RES with reducible fibers $I_4^*+2I_1$ by considering the $I_1$-fibers as branch fibers (see also Table \eqref{table extremal}). If we denote by $\iota_{1,1}$ the deck transformation, we observe the fibration with fibers $2I_4^*+2I_2$ is of type 2 with respect to $\iota_{1,1}$.

\section{Extremal elliptic fibrations}\label{Sec: extremal}
We consider special K3 surfaces, namely those obtained as a quadratic base change of extremal rational elliptic surfaces, studying a problem analogous to the one treated in \cite[Section 7]{MP}.
\begin{definition}
	An elliptic fibration $\mathcal{E}_Z:Z\ra\mathbb{P}^1$ on a surface $Z$ is called extremal if $rk(MW(\mathcal{E}_Z))=0$ and the Picard number is the largest possible, i.e. $rk(Pic(Z))=h^{1,1}(Y)$.
\end{definition}
In particular, a RES is extremal if the rank of the Mordell--Weil group is 0. An elliptic K3 surface is extremal if the rank of the Mordell--Weil group is 0 and its Picard number 20.

The extremal rational elliptic surfaces are classified (see \cite{MP}) and fit in a list of 16 cases.

The extremal elliptic K3 surfaces are classified (see \cite{SZ}) and fit in a list of 325 cases.

\begin{rem} An  extremal RES contains a finite number of negative curves. This allows one to classify all the conic bundles on it, see \cite{GSconcis}.  Since the conic bundles on RES induce elliptic fibrations on the K3 surface obtained by a quadratic base change, it seems of particular interest to consider quadratic base changes on extremal RES. 

An extremal elliptic K3 surfaces could contain infinitely many negative curves.\end{rem}

In \cite[Section 7]{MP}, the authors determine what extremal rational elliptic surfaces can be obtained as a base change of another extremal RES. Here we reconsider the analogous problem, namely, which extremal elliptic K3 surfaces can be obtained as a quadratic base change of an extremal RES?

\begin{proposition}
	A quadratic base change $\psi:\mathbb{P}^1\ra\mathbb{P}^1$ of a RES $\mathcal{E}_R:R\ra\mathbb{P}^1$ as in diagram \eqref{diag base chage} produces an extremal K3 surface if, and only if, the following conditions are satisfied:
	\begin{itemize}
		\item[a)] $\mathcal{E}_R:R\ra\mathbb{P}^1$ is an extremal RES;
		\item[b)] $\psi$ is branched along singular reduced fibers of $\mathcal{E}_R$.
	\end{itemize}
	There are 25 inequivalent extremal elliptic fibrations on K3 surfaces obtained as a quadratic base change by a RES.
\end{proposition}
\begin{proof}
	Let $X$ be an extremal elliptic K3 that is a quadratic base change of a RES $R$ as in diagram \eqref{diag base chage}. Since every section of $\mathcal{E}_R$ induces a section of $\mathcal{E}_X$, and $\rk(MW(\mathcal{E}_X))=0$, we must have $\rk(MW(\mathcal{E}_R))=0$ and hence $\mathcal{E}_R$ has to be extremal.
	 
	 If the base change was branched on two smooth fibers, then $\mathcal{E}_X$  has twice as many reducible fibers as $R$. The lattice generated by the non-trivial components of the reducible fibers of $\mathcal{E}_R$ has rank 8 (since $\mathcal{E}_R$ is extremal and $\rho(R)=10$), so the one generated by the non-trivial components of the reducible fibers of $\mathcal{E}_R$ has rank 16. Hence $\rho(X)=2+16=18$, which contradicts the fact that $\mathcal{E}_X$ is extremal. Similarly one excludes the case in which only one of the branch fibers is singular. 
	We recall that the branch fibers are necessarily reduced, otherwise $X$ would not be a K3 surface.
	
	Vice versa, by a case-by-case analysis, one can directly check that if $\mathcal{E}_R$ is an extremal RES and one considers a base change branched over two singular reduced fibers one obtains an elliptic K3 surface $\mathcal{E}_X:X\ra\mathbb{P}^1$ such that the trivial lattice has rank 20. This implies that $\rho(X)=20$ and $\rk(\mathcal{E}_X)=0$.
	
	This case-by-case analysis produces the following table, which allows one also to determine all the extremal elliptic K3 surfaces that can be obtained as a quadratic base change of a RES.
	
	In the first column there are all the extremal rational elliptic surfaces that have at least two singular reduced fibers, in the second the admissible choice for the branch fibers, in the third the singular fibers of $X$, which completely determines the transcendental lattice of $X$, given in the fourth column and obtained by comparing the fibrations with the one listed in \cite{SZ}. The transcendental lattice is represented by a matrix in the form $\left[\begin{array}{cc}a&b\\b&c\end{array}\right]$ and we express it as $(a,b,c)$. 
	
	\begin{equation}\label{table extremal}\begin{array}{cccccc}\#&\mbox{ singular fibers }\mathcal{E}_R&\mbox{branch fibers}&\mbox{ singular fibers }\mathcal{E}_X&T_X\\
		1&II^*+2I_1&2I_1&2II^*+2I_2&(2,0,2)\\
		2&I_4^*+2I_1&2I_1&2I_4^*+2I_2&(2,0,2)\\
		3&I_9+3I_1&2I_1&2I_9+2I_2+2I_1&(4,2,10)\\
		4&I_9+3I_1&I_9+I_1&I_{18}+I_2+4I_1&(2,0,2)\\
		5&III^*+I_2+I_1&I_2+I_1&2III^*+I_4+I_2&(2,0,4)\\
		6&I_8+I_2+2I_1&2I_1&2I_8+4I_2&(4,0,4)\\
		7&I_8+I_2+2I_1&I_2+I_1&2I_8+I_4+I_2+2I_1&(2,0,4)\\
		8&I_8+I_2+2I_1&I_8+I_1&I_{16}+3I_2+2I_1&(2,0,4)\\
		9&I_8+I_2+2I_1&I_8+I_2&I_{16}+I_4+4I_1&(2,0,2)\\
		10&IV^*+I_3+I_1&I_3+I_1&2IV^*+I_6+I_2&(2,0,6)\\
		11&I_2^*+2I_2&2I_2&2I_2^*+2I_4& (4,0,4)\\
		12&I_1^*+I_4+I_1&I_4+I_1&2I_1^*+I_8+I_2& (2,0,8)\\
		13&I_6+I_3+I_2+I_1&I_2+I_1&2I_6+I_4+2I_3+I_2& (6,0,12)\\
		14&I_6+I_3+I_2+I_1&I_3+I_1&3I_6+3I_2&(2,0,6)\\
		15&I_6+I_3+I_2+I_1&I_6+I_1&I_{12}+2I_3+3I_2&(2,0,12)\\
		16&I_6+I_3+I_2+I_1&I_3+I_2&3I_6+I_4+2I_1&(4,0,6)\\
		17&I_6+I_3+I_2+I_1&I_6+I_2&I_{12}+I_4+2I_3+2I_1&(4,2,4)\\
		18&I_6+I_3+I_2+I_1&I_6+I_3&I_{12}+I_6+2I_2+2I_1&(2,0,4)\\
		
		19&2I_5+2I_1&2I_1&4I_5+2I_2&(10,0,10)\\
		20&2I_5+2I_1&I_5+I_1&I_{10}+2I_5+I_2+2I_1&(2,0,10)\\
		21&2I_5+2I_1&2I_5&2I_{10}+4I_1&(2,0,2)\\
		22&2I_4+2I_2&2I_2&6I_4&(4,0,4)\\
		23&2I_4+2I_2&I_4+I_2&I_8+3I_4+2I_2&(4,0,8)\\
		24&2I_4+2I_2&2I_4&2I_8+4I_2&(4,0,4)\\
		25&4I_3&2I_3&2I_6+4I_3&(6,0,6)\\
	\end{array}\end{equation}
	
	One observes that cases 1,2,4,9,21 correspond to the same K3 surface since the transcendental lattice identifies the K3 surface uniquely. The same holds true for cases 5,7,8, for cases 6,11,22,24, and for cases 10 and 14. Hence we obtained 15 K3 surfaces, some of which admit more than one extremal elliptic fibration.
\end{proof}

\begin{rem} The fact that on the same K3 surface, there are inequivalent elliptic fibrations (in particular extremal elliptic fibrations) which are obtained as a base change by a RES, implies that there are fibrations of type 2 with respect to the involution related with one of these base changes. More explicitly, we can interpret the K3 surface $X$ as a base change of a specific RES. In particular, $X$ is endowed with the cover involution $\iota$. All the other inequivalent elliptic fibrations on $X$ obtained as base changes from other rational elliptic surfaces are of type 2 with respect to $\iota$.\end{rem}

We observe that in \cite{WINEequ} we considered the base change in case 21 of the previous table, denoted by $S_{5,5}$ in \cite{WINEequ}.
The cases 1,2,4,9,21 correspond to a K3 surface with a 2-elementary N\'eron--Severi lattice. The elliptic fibrations for the latter are classified (see eg. \cite{CG}).

\begin{rem}
	In \cite[Theorem 7.3]{MP}, the authors proved that the extremal RES with singular fibers $I_8+I_2+2I_1$ and $2I_4+2I_2$ can be obtained by a quadratic base change on another extremal RES. Composing the quadratic base change described in \cite[Theorem 7.3]{MP} with the one considered in Table \ref{table extremal}, one obtains a $4:1$ base change from an extremal RES which produces an extremal elliptic K3 surface. The cover group can be either $(\Z/2\Z)^2$ or $\Z/4\Z$ and both these cases appear (for example the K3 surface in line 6 can be obtained as $(\Z/2\Z)^2$ cover of the RES with singular fibers $I_4^*+2I_1$, the one in line 9 as $\Z/4\Z$ cover of the same RES).
\end{rem}
\section{K3s as double cover of minimal rational surfaces}\label{sec: minimal}
As in the previous two sections, we consider elliptic K3 surfaces obtained as a quadratic base change of a RES $\mathcal{E}_R:R\ra\mathbb{P}^1$. In what follows, we present a different point of view on the same surfaces, namely, such K3 surfaces can be naturally realized as a double cover of a minimal rational surface whose branch locus satisfies certain properties. We first describe the relationship between these different points of view and then show that their interactions allow one to write explicit equations for the elliptic fibrations on $X$, following \cite{WINEequ}.

\begin{proposition}\label{prop: minimal model of RES}
	Let $\mathcal{E}_R:R\rightarrow \mathbb{P}^1$ be a relatively minimal RES. Let $\beta:R\ra R^{o}$ be a composition of contractions of $(-1)$-curves to a minimal surface $R^{o}$, then $R^{o}$ is either $\mathbb{P}^2$ or $\mathbb{P}^1\times\mathbb{P}^1$ or $\mathbb{F}_2$. In the latter case $\mathcal{E}_R:R\rightarrow\mathbb{P}^1$ admits at least one reducible fiber.
	\end{proposition}
\begin{proof} 
Since there are no curves on $R$ with self-intersection less than $-2$, also on $R^{o}$ there cannot be curves with self-intersection less than $-2$. This excludes all the Hirzebruch surfaces $\mathbb{F}_n$ with $n\geq 3$ as possible minimal models of $R$. For the same reason if $R^{o}\simeq \mathbb{F}_2$, then $R$ contains a $(-2)$-curve, and hence it admits a reducible fiber. The surface $\mathbb{F}_0$ is $\mathbb{P}^1\times \mathbb{P}^1$ and the surface $\mathbb{F}_1$ is not minimal.
\end{proof}

\begin{proposition} 
Let $X$ be obtained as a quadratic base change by a RES $\mathcal{E}_R:R\rightarrow \mathbb{P}^1$ as in diagram \eqref{diag base chage}. Let $\beta:R\rightarrow R^{o}$ be the birational morphism described in Proposition \ref{prop: minimal model of RES}. Then $X$ is the minimal resolution of a double cover of $R^{o}$ branched on a reducible curve with at least two irreducible components.
\end{proposition}
\begin{proof}
By diagram \eqref{diag base chage}, we obtain a generically $2:1$ map $q_X:X\ra R$ branched along two fibers of $R$. The birational morphism  $\beta:R\rightarrow R^o$ is the blow up of $R^o$ in the base points of a pencil $\mathcal{P}$ of genus 1 curves and so $\beta$ maps two fibers of $\mathcal{E}_R$ to two curves in this pencil. Hence $\beta\circ q_X:X\ra R^o$ is a generically $2:1$ map branched on two curves in the pencil $\mathcal{P}$.\end{proof}
We consider a reciprocal of the Proposition \ref{prop: minimal model of RES} which allows one to interpret certain double covers of minimal rational surfaces as base change of rational elliptic surfaces.
To state the result, we denote: by $H$ the class of a line in $\mathbb{P}^2$; by $h_1$, $h_2$ the generators of $\mathrm{Pic}(\mathbb{P}^1\times\mathbb{P}^1)$ such that $h_1^2=h_2^2=0$ and $h_1\cdot h_2=1$; by $\Gamma$, $\Phi$ the generators of $\mathrm{Pic}(\mathbb{F}_n$ such that $\Phi^2=0$, $\Gamma^2=-n$, $\Phi\Gamma=1$.

\begin{corollary}\label{cor: double cover minimal rat are base change} Let $X$ be a K3 surface and $X\ra W$ be a double cover of a minimal rational surface $W$ branched on the union of two (possibly reducible) curves $C_1$ and $C_2$. If $W=\mathbb{P}^2$ (resp. $\mathbb{P}^1\times \mathbb{P}^1$, $\mathbb{F}_2$) and $C_1\sim C_2\in |3H|$, (resp. $|2h_1+2h_2|$, $|2\Gamma+4\Phi|$), then $X$ is obtained by a quadratic base change on the RES obtained blowing up $W$ in the points $C_1\cap C_2$.\end{corollary}
\begin{proof}
By \cite[Chapter V Corollary 2.18]{Hartshorne}, the linear system $|2\Gamma+4\Phi|$ contains smooth curves. We observe that all linear systems in the statement are the anticanonical linear system in the respective surface, and hence the genus of a general member is 1. Hence it suffices to consider two smooth curves in such linear systems to find the required pencil on the minimal surface $W$.
\end{proof}

\begin{remark}
We observe that:
\begin{itemize}
\item[i)] by choosing $C_1$ and $C_2$ general in their linear system in $\mathbb{P}^2$ or $\mathbb{P}^1\times \mathbb{P}^1$ (resp. $\mathbb{F}_2$), one can recover a general  RES with no reducible fibers (resp. with exactly one reducible fiber). The other configurations of singular fibers can then be obtained by specializing $C_1$ and $C_2$.
\item[ii)] explicit examples of contractions of the RES to different minimal models are provided in \cite[Figures 1 and 3]{WINEfield}.
\end{itemize}
\end{remark}

The following example deals with a well-known case and shows the relevance of Corollary \ref{cor: double cover minimal rat are base change}.

 \begin{example}
 {\rm Let $X$ be the double cover of $\mathbb{P}^2$ branched on six lines in general position $\ell_1,\ldots, \ell_6$. The elliptic fibrations on $X$ are classified by Kloosterman, \cite{Kl}. In order to show the existence of the elliptic fibrations he considered pencils of curves in $\mathbb{P}^2$ whose pullback to $X$ give pencils of genus 1 curves. A slightly different point of view is the following: let us consider the curves $C_1=\ell_1\cup \ell_2\cup \ell_3$ and $C_2=\ell_4\cup \ell_5\cup \ell_6$. The pencil generated by $C_1$ and $C_2$ is a pencil of genus 1 curves and the blow-up of $\mathbb{P}^2$ in its base point yields a RES $\mathcal{E}_R:R\ra\mathbb{P}^1$ with two fibers of type $I_3$ (corresponding to the curves $C_1$ and $C_2$). By Corollary \ref{cor: double cover minimal rat are base change}, $X$ is the base change of $\mathcal{E}_R:R\ra\mathbb{P}^1$ branched on the two $I_3$-fibers, so $X$ is naturally endowed with an elliptic fibration with two fibers $I_6$ (the ones obtained as a base change of $R$ branched on the two $I_3$ fibers). The other elliptic fibrations on $X$ are either induced by a generalized conic bundle (classified in \cite[Section 9]{Kl}) or by splitting genus 1 pencil (classified in \cite[Section 8]{Kl}). }

\end{example}
The described point of view, i.e. consider K3 surfaces which are a double cover of minimal rational surfaces and elliptic fibrations induced by linear systems of plane curves, is the one adopted in \cite{Kl} and \cite{CG}. In both these papers the surface $R^{o}$ is isomorphic to $\mathbb{P}^2$. In \cite{WINEfield}, we also considered elliptic fibrations related with a model of $X$ as double cover of $R^{o}\simeq \mathbb{P}^1\times\mathbb{P}^1$ and $R^o\simeq \mathbb{F}_2$ (see e.g. \cite[Figures 1 and 3]{WINEfield}). In particular, in \cite{WINEfield} one considers elliptic fibrations whose field of definition is not necessarily algebraically closed; in this case, it is possible that $R$ cannot be contracted over the field of definition to $\mathbb{P}^2$, but it can be contracted to some other minimal models of rational surfaces.

\subsection{An algorithm to determine Weierstrass equations}
$ $

In \cite{WINEequ}, the authors consider certain K3 surfaces obtained as a quadratic base change of a RES, classify conic bundles $|B|$, and construct generalized conic bundles on the RES. They consider the push-forward of these conic bundles to $R^o$ obtaining pencils of curves in $\mathbb{P}^2$, hence intertwining the two points of view. We observe that the same can be reproduced for $R^o= \mathbb{P}^1\times\mathbb{P}^1$ or $\mathbb{F}_2$. The advantage of this point of view is that the classification of the conic bundles is quite easy, in particular when one has a strong control on the negative curves on $R$, as in the case of extremal rational elliptic fibrations. On the other hand, one is able to explicitly write equations for pencils of plane curves (or of curves in $\mathbb{P}^1\times\mathbb{P}^1$) with some prescribed properties of the base points. This allows the authors to apply an algorithm that gives the Weierstrass equation of some of the classified elliptic fibrations on $X$.

By Corollary \ref{cor: double cover minimal rat are base change}, we know that $X$ is a double cover of $\mathbb{P}^2$ branched on the union of two cubics, so its equation is of the form
\begin{equation}\label{eq: V double cover P2}w^2=f_3(x_0:x_1:x_2)g_3(x_0:x_1:x_2).\end{equation}
Let $B$ be a conic bundle on $R$, i.e., a basepoint-free linear
system of rational curves giving $R \rightarrow\mathbb{P}^1_\tau$.
Pushing forward to $\mathbb{P}^2$, $B$ is given by a pencil of
plane rational curves with equation $h(x_0:x_1:x_2,\tau)=0$.  The
polynomial $h(x_0:x_1:x_2,\tau)$ is homogeneous in $x_0$, $x_1$,
$x_2$, say of degree $e\geq 1$ and linear in $\tau$.

For general
$\tau$, we must find an isomorphism of the curve
$h(x_0:x_1:x_2,\tau) =0$ with $\mathbb{P}^1$, and extract the
images of the four intersection points with $f_3g_3 =0$.

When $e \leq 3$, an isomorphism with $\mathbb{P}^1$ is provided by
projection from a point of order $e-1$ on the curve (e.g. any
point in $\mathbb{P}^2$ if $e=1$, a point on the conic if $e=2$,
and a double point of the cubic if $e=3$).  Such a point
necessarily exists (in the case $e=3$, the singularity must be a
basepoint of the pencil) and is also necessarily a basepoint of
the original pencil of cubics giving $\E_R$.  Up to acting by
$\PGL_3(\C)$, we may assume that this point is $(0:1:0)$.

\noindent\textbf{Algorithm when $e\leq 3$.}
\begin{enumerate}
	\item Compute the resultant of the polynomials
	$f_3(x_0:x_1:x_2)g_3(x_0:x_1:x_2)$ and $h(x_0:x_1:x_2,\tau)$ with
	respect to the variable $x_1$. The result is a polynomial
	$r(x_0:x_2,\tau)$ which is homogeneous in $x_0$ and $x_2$,
	corresponding to the images of \textit{all} the intersection
	points $\{f_3g_3=0 \} \cap \{h_\tau=0\}$ after projection from
	$(0:1:0)$.
	
	\item Since $B$ is a conic bundle, $r(x_0:x_2,\tau)$ will be of
	the form $a(x_0:x_2,\tau)^2b(x_0:x_2,\tau)c(\tau)$, where $a$ and
	$b$ are homogeneous in $x_0$ and $x_2$, the degree of $a$ depends
	upon $e$ and the degree of $b$ in $x_0$ and $x_2$ is 4. \item The
	equation of $V$ is now given by $w^2=r(x_0:x_2,\tau)$, which is
	birationally equivalent to
	\begin{equation}\label{eq: genus 1 fibration by conic bundles}w^2=c(\tau)b(x_0:x_2,\tau),\end{equation} by the change of coordinates
	$w\mapsto wa(x_0:x_2,\tau)$. Since for almost every $\tau$, the
	equation \eqref{eq: genus 1 fibration by conic bundles} is the
	equation of a $2:1$ cover of $\mathbb{P}^1_{(x_0:x_2)}$ branched
	in 4 points, \eqref{eq: genus 1 fibration by conic bundles} is the
	equation of the genus 1 fibration on the K3 surface $V$ induced by
	the conic bundle $|B|$. \item If there is a section of fibration
	\eqref{eq: genus 1 fibration by conic bundles}, then it is
	possible to obtain the Weierstrass form by standard
	transformations.\end{enumerate}

\begin{rem} The algorithm is not restricted to conic bundles. It can be applied exactly in the same way to the generalized conic bundles.\end{rem}

When $e\geq 4$, the projection from a point may suffice, for example,
if all curves have a basepoint of degree $e-1$. However there are
several conic bundles whose general members cannot be parametrized
by lines and in this case, one has to consider conics. In \cite[Section 5.2]{WINEequ}, an algorithm is presented also in this case. Moreover, also in the case of elliptic fibrations induced by splitting genus 1 pencils, similar algorithms exist, see \cite[Section 6.1.2]{WINEequ}.

\end{document}